\documentclass[10pt,twoside]{amsart}
\usepackage{amsmath,amsthm,amscd,amssymb, amsfonts}
\usepackage{enumerate}
\usepackage{latexsym}
\usepackage{amstext}

\newtheorem{teo}{Theorem}[section]
\newtheorem{lem}[teo]{Lemma}
\newtheorem{pro}[teo]{Proposition}

\newtheorem{cor}[teo]{Corollary}

\theoremstyle{definition}
\newtheorem{defi}[teo]{Definition}

\newtheorem{remark}[teo]{Remark}

\newcommand{\D}{\mathbb{D}}

\begin{document}

\title{Toeplitz Operators on Weighted Bergman Spaces}
\author{Gerardo R. Chac\'on}
\address{Departamento de Matem\'aticas, Pontificia Universidad Javeriana, Bogot\'a, Colombia}
\email{chacong@javeriana.edu.co}
\maketitle
\begin{abstract}
In this article we characterize the boundedness and compactness of a Toeplitz-type operator on weighted Bergman spaces satisfying the so-called Bekoll\'e-Bonami condition in terms of the Berezin transform.
\end{abstract}

\vskip.2in

\section{Introduction}
Let $\D$ denote the unit disc in the complex plane and $dA$ the normalized Lebesgue measure on $\D$. For a given nonnegative integrable function $u$ on $\D$ we define the {\it weighted Bergman space} $A^2(u)$ as the space of all analytic functions on $\D$ that belong to the weighted space $L^2(u)$. That is, $f\in A^2(u)$ if $f$ is analytic on $\D$ and satisfies: $$\|f\|_{A^2(u)}^2:=\int_\D |f(z)|^2u(z)dA(z).$$

Weighted Bergman spaces have been studied by several authors in different contexts (see for example \cite{Ha},\cite{OP},\cite{St}, and \cite{Z}). Most of the research about Toeplitz-type operators in this spaces have been done by considering radial weights. In this article we will consider weights satisfying the the so-called B\'ekoll\'e-Bonami condition.

\begin{defi}
 A function $u$ in $L^1(\D)$ is said to satisfy the Bekolle-Bonami condition $B_2$ if there exists a constant $C>0$  such that $$\frac{1}{(A(S(I)))^2}\int_{S(I)} u(z)dA(z)\int_{S(I)} \frac{1}{u(z)}dA(z)\leq C $$ for every interval $I\subset \partial\D$. Here, $S(I)$ denotes the {\it Carleson square}: $$S(I)=\{re^{it}:e^{it}\in I, \quad 1-\frac{|I|}{2\pi}\leq r<1\}.$$

\end{defi}

 Condition $B_2$ was introduced by Bekollé and Bonami in \cite{BB}. They showed that $B_2$ is necessary and sufficient for the Bergman projection to be bounded on $L^2(udA)$. Condition $B_2$ was used by Luecking in \cite{Lu2, Lu} to study Carleson measures in weighted Bergman spaces. Then in \cite{NY} Nakazi and Yamada generalized Luecking's results by introducing a more restrictive condition $(A_2)_\partial$ and found, under this condition, a characterization for Carleson measures on $A^2(u)$ spaces. They also gave several examples of weights satisfying this condition. In \cite{C2}, Constantin generalized Luecking's result and studied Toeplitz-type operators on spaces $A^2(u)$. The characterization of Carleson measures is as follows.

\begin{teo}[\cite{Lu2}, \cite{C2}]\label{t2}
Suppose $u>0$ a.e. in $\D$ and satisfies the $B_2$ condition, let $d\mu=udA$ and let $\nu$ be a positive
Borel measure on $\D$. Then the following are equivalent.
\begin{enumerate}
\item There exist a constant $C>0$ such that $$\int_{\D} |f|^2 d\nu \leq C\int_{\D} |f|^2
d\mu$$ for every polynomial $f$ i.e. $\nu$ is a $\mu$-Carleson measure.

\item There exists $r>0$ and $\gamma>0$ satisfying $$\nu(D_r(a))\leq\gamma
\mu(D_r(a))$$ for all $a\in \D$.

\end{enumerate}
\end{teo}

\begin{teo}[\cite{Lu}]\label{t3}
Suppose $u\in L^1(dA)$ satisfies condition $B_2$, then the dual of $A^2(u)$ can be identified with $A^2(u^{-1})$. The pairing is given by $$\langle f,g \rangle = \int_{\D} f\overline{g}dA$$
\end{teo}

\begin{remark}
Notice that if $u$ satisfies the $B_2$ condition, then it is not hard to show (see for example \cite{Lu}) that there is a constant $C>0$ such that for any function $f$ analytic in $\D$ and any $a\in \D$, it holds that $$|f(a)|^2 \leq C\left(\int_{D_r(a)} udA\right)^{-1} \|f\|_{A^2(u)}^2.$$ Consequently, evaluation functionals are bounded on $A^2(u)$ and so for every complex number $a\in \D$ there exists a function $K_u(\cdot,a) \in A^2(u)$ such that $$\langle f, K_u(\cdot, a) \rangle_{A^2(u)} = f(a)$$ for every function $f\in A^2(u)$; i.e. the functions $K_u(\cdot,a)$ are the {\it reproducing kernels} for $A^2(u)$. Notice also that the set $\{K_u(\cdot, w):w\in\D\}$ is dense in $A^2(u)$.

Moreover, using theorem \ref{t3}, we have that there exists a bounded, bijective, linear operator $F_1:A^2(u^{-1})\to(A^2(u))^\ast$ such that $F_1(f)(g)=\langle g, \overline{f} \rangle_{A^2}$ for every $f\in A^2(u^{-1})$ and $g\in A^2(u)$. On the other hand, since $A^2(u)$ is a Hilbert space, there exists a bounded, bijective, linear operator $F_2:(A^2(u))^\ast\to A^2(u)$ such that $\langle g, \overline{F_2 (\gamma)}\rangle_{A^2(u)} = \gamma(g)$ for every $\gamma\in (A^2(u))^\ast$ and $g\in A^2(u)$.

Let $\gamma_a\in (A^2(u))^\ast$ denote the evaluation functional at $a$; then $F_2(\gamma_a)=K_u(\cdot,a)$. On the other hand, if we let $B_a(z)$, with $a,z\in\D$, denote the {\it Bergman kernel} \[B_a(z):=\frac{1}{(1-\overline{a}z)^2},\] then since $B_a\in A^2(u^{-1})$, and the pairing is given by the $A^2$ product, we have that $F_1(B_a)=\gamma_a$.

Define $T:A^2(u)\to A^2(u^{-1})$ by $T=F_1^{-1}F_2^{-1}$. Then $T$ is a bounded operator and $T(K_u(\cdot,a))=F_1^{-1}(\gamma_a)=B_a$. Also,  the operator $T^{-1}:A^2(u^{-1})\to A^2(u)$ is well
defined and consequently we have that $$\|B_a\|_{A^2(u^{-1})}\sim\|K_u(\cdot,a)\|_{A^2(u)}.$$
\end{remark}

\begin{defi}
Let $\nu$ be a positive Borel measure on $\D$. Define for each polynomial $p$ the Toeplitz operator $T_\nu$ $$T_\nu p (z):= \int_{\D} K_u(z,w)p(w)d\nu(w).$$
\end{defi}

We will study the conditions under which Toeplitz operators can be extended to $A^2(u)$ and we will characterize the boundedness and compactness of Toeplitz operators acting on the $A^2(u)$ spaces in terms of the so-called Berezin trasform of the measure $\nu$.

\begin{defi}
Let $\nu$ be a positive Borel measure on $\D$, the Berezin transform $\tilde{\nu}$ of $\nu$ is defined as $$\tilde{\nu}(a):=\frac{1}{\|B_a\|_{A^2(u)}}\int_{\D} |B_a(z)|^2d\nu(z).$$

\end{defi}

From now on, we will assume that $\nu$ is a probability measure.

  The following result is well known (see \cite{Ha}, \cite{OP},\cite{St} and \cite{Z}):

\begin{teo}\label{t1}
Supppose $\mu$ is a finite positive Borel measure on $\D$. Then the following are equivalent (here, $dA_\alpha(z)$ represents the standard weighted Lebesgue measure $(1-|z|^2)dA(z)$, $\alpha>-1$):
\begin{itemize}
\item[(a)] $T_\mu$ is bounded on $A^2(dA_\alpha)$.
\item[(b)] $\tilde{\mu}$ is a bounded function on $\D$.
\item[(c)] $\mu$ is a Carleson measure for $A^2(dA_\alpha)$
\end{itemize}
\end{teo}

Where, $$\tilde{\mu}(z)=\int_{\D} |b_z(w)|^2 d\mu(w), \qquad z\in \D$$ is the Berezin transform of the Toeplitz operator, here $$b_z(w):=\frac{1-|z|^2}{(1-\overline{z}w)^2}.$$

The proof of this result mainly relies on the behavior of the reproducing kernels $b_z(w)$ when $z$ is close to $w$ and on the characterization of Carleson measures for the Bergman space. In the case of $A^2(u)$ spaces, such a characterization is known for Carleson measures (see for example \cite{C2}) but we do not have an explicit formula for the reproducing kernels, consequently we will need to use a different technique. This will be an atomic decomposition developed in terms of the reproducing kernels of $A^2(u)$.

\section{Atomic decomposition for $A^2(u)$}

In this section, we develop a way of expressing functions in $A^2(u)$ as a linear combination of reproducing kernels. Similar problems have been studied in \cite{C} and \cite{Lu}. We will use a similar reasoning as in \cite{Lu}, where the problem is studied in terms of the Bergman kernel. We will assume in the rest of this article that $\mu :=udA$ is a probability measure.

\begin{defi}
Let $\varepsilon>0$. A sequence $\{a_n\} \subset \D$ is called $\varepsilon$-separated if $\inf\{\rho (a_n, a_m):n\neq m\}\geq \varepsilon >0$. Here, $\rho$ denotes the pseudohyperbolic metric: $\rho(z,w):=\displaystyle \frac{|z-w|}{|1-\overline{z}w|}$.
\end{defi}

\begin{teo}[\cite{Lu}]\label{tludecomp}
If $u$ satisfies condition $B_2$, then there exists an $\varepsilon$-separated sequence $\{a_n\}$ such that $$\|f\|^2_{A^2(u)}\sim \sum_k |f(a_n)|^2\mu(D_r(a_n))$$ for every function $f\in A^2(u)$.
\end{teo}

\begin{teo}\label{Thmatomdecomp}
Suppose $u$ satisfies the condition $B_2$. Then there exists a sequence $\{a_n\}\subset \D$ which is $\varepsilon$-separated for some $\varepsilon>0$, such that any $f\in A^2(u)$ has the form
\begin{equation}\label{eqdecomp}
f(z)=\sum_n c_n K_u(\cdot, a_n)\mu(D_r(a_n))^{1/2}
\end{equation}
for some sequence $\{c_n\} \in l^2$
\end{teo}

\begin{proof}
By theorem \ref{tludecomp}, we know that there exists an $\varepsilon$-separated sequence $\{a_n\}$ such that $\|f\|^2_{A^2(u)}\sim \sum_k |f(a_n)|^2\mu(D_r(a_n))$. Consider the linear operator $R:A^2(u)\to l^2$ defined by: $$Rf:=\{f(a_n)\mu(D_r(a_n))^{1/2}\}.$$ Then, $\|Rf\|_{l^2}^2=\displaystyle\sum_n |f(a_n)|^2\mu(D_r(a_n)) \sim \|f\|_{A^2(u)}^2$ and so $R$ is a bounded injective linear operator having closed range. Consequently, the dual mapping $R^\ast :l^2\to A^2(u)$ is onto. Now, let $\{c_n\} \in l^2$, then for $f\in A^2(u)$ we have

\begin{eqnarray*}
\langle R^\ast \{c_n\} , f \rangle_{A^2(u)} &=& \langle \{c_n\}, Rf\rangle_{l^2}\\
                                            &=& \sum_n c_n\overline{f(a_n)}\mu(D_r(a_n))^{1/2}\\
                                            &=& \sum_n c_n\int_{\D} \overline{f(z)}K_u(z,a_n)d\mu(z) \mu(D_r(a_n))^{1/2}\\
                                            &=& \int_{\D} \overline{f(z)}\sum_n c_n K_u(z,a_n)\mu(D_r(a_n))^{1/2}d\mu(z)\\
                                            &=& \left\langle \sum_n a_n K_u(\cdot,a_n)\mu(D_r(a_n))^{1/2}, f\right\rangle_{A^2(u)}.
\end{eqnarray*}
Here, we justify the interchange of integration and summation as follows: Given a sequence $\{c_n\}\in l^2$ and $\varepsilon>0$, there exists $N\in \mathbb{N}$ such that if $m,l\geq N$ then $$\sum_{n=m}^l|c_n|^2<\displaystyle\frac{\varepsilon}{\|R^\ast\|}.$$ Thus,
\begin{eqnarray*}
&&\sum_{n=m}^l c_n\int_{\D}\overline{f(z)}K_u(z,a_n)d\mu(z)\mu(D_r(a_n))^{1/2}\\
&&\qquad= \int_{\D}\overline{f(z)}\sum_{n=m}^l c_n K_u(z,a_n)\mu(D_r(a_n))^{1/2}d\mu(z)
\end{eqnarray*}
and
\begin{eqnarray*}
&&\left\|\sum_{n=m}^l c_n K_u(z,a_n)\mu(D_r(a_n))^{1/2}\right\|_{A^2(u)}\\ &&\qquad=  \sup_{\|f\|_{A^2(u)}=1}\left|\left\langle \sum_{n=m}^l c_n K_u(z,a_n)\mu(D_r(a_n))^{1/2}, f\right\rangle_{A^2(u)}\right|\\
          &&\qquad=  \sup_{\|f\|_{A^2(u)}=1}\left| \sum_{n=m}^l c_n\int_{\D} K_u(z,a_n)d\mu(z)\mu(D_r(a_n))^{1/2} \right|\\
          &&\qquad= \sup_{\|f\|_{A^2(u)}=1}|\langle R^\ast\{\hat{c}_n\},f\rangle_A^2(u)|\leq \|R^\ast\|\|\{\hat{c}_n\}\|_{l^2}<\varepsilon
\end{eqnarray*}
where $$\hat{c}_n:=\left\{
                     \begin{array}{ll}
                       c_n, & \hbox{if } n\in\{m,\dots,l\} \\
                       0, & \hbox{otherwise}
                     \end{array}
                   \right.
$$

Thus, $s_m:=\sum_{n=1}^m c_n K_u(\cdot,a_n)\mu(D_r(a_n))^{1/2}$ forms a Cauchy sequence and consequently it converges in the $A^2(u)$-norm to \[s:=\sum_{n=1}^\infty c_n K_u(\cdot,a_n)\mu(D_r(a_n))^{1/2}\in A^2(u).\] Therefore, $$\int_{\D}f(z)\overline{s_m(z)}d\mu(z)\to \int_{\D}f(z)\overline{s(z)}d\mu(z).$$

Hence, $R^\ast\{c_n\}=\sum_{n=1}^\infty c_n K_u(\cdot,a_n)\mu(D_r(a_n))^{1/2}$ and since $R^\ast$ is surjective, then the result holds.

\end{proof}

\section{Boundedness of Toplitz operators}

\begin{teo}\label{tboundbere}
Suppose that $u$ satisfies the $B_2$ condition and that the Toeplitz operator $T_\nu :A^2(u)\to A^2(u)$ is bounded. Then the Berezin transform $\tilde{\nu}$ is bounded.
\end{teo}

\begin{proof}
First, notice that for every finite sum $s_m:=\displaystyle\sum_{n=1}^m\lambda_n K_u(\cdot,a_n)$, $a_n\in\D$, the following holds:
\begin{eqnarray*}
&&\left\langle T_\nu\left(\sum_{n=1}^m\lambda_n K_u(\cdot,a_n)\right),\sum_{n=1}^m\lambda_n K_u(\cdot,a_n)\right\rangle_{A^2(u)} \\&&= \sum_{l=1}^m T_{\nu}\left(\sum_{n=1}^m\lambda_n K_u(\cdot,a_n)\right)(a_l)\\
&&= \sum_{l=1}^m \lambda_l \int_{\D} K_u(a_l,z)\sum_{n=1}^m\lambda_n K_u(z,a_n)d\nu(z)\\ &&= \left\|\sum_{n=1}^m\lambda_n K_u(\cdot,a_n)\right\|_{L^2(\nu)}^2\\ &&\lesssim \left\|\sum_{n=1}^m\lambda_n K_u(\cdot,a_n)\right\|_{A^2(u)}^2.
\end{eqnarray*}
So, if $\{s_m\}$ is a Cauchy sequence in $A^2(u)$, it is also a Cauchy sequence in $L^2(\nu)$ and so it converges in $L^2(\nu)$. Therefore, by theorem \ref{Thmatomdecomp} we have that if $s_m\to g$ in the $A^2(u)$-norm, then $s_m\to g$ in $L^2(\nu)$ and consequently for any function $f\in A^2(u)$, $\langle f,s_m\rangle_{L^2(\nu)}\to\langle f, g\rangle_{L^2(\nu)}$. Moreover,
\begin{eqnarray*}
\langle T_\nu f, g\rangle_{A^2(u)}&=&\lim_{m\to\infty}\langle T_\nu f, s_m\rangle_{A^2(u)}\\&=&\langle f, s_m\rangle_{L^2(\nu)}.
\end{eqnarray*}
Therefore,
\begin{equation}\label{toepeq1}
\langle T_\nu f,g\rangle_{A^2(u)}=\langle f, g\rangle_{L^2(\nu)}.
\end{equation}
Hence for any $a\in \D$,
\begin{eqnarray*}
\langle T_\nu B_a, B_a\rangle_{A^2(u)} &=& \langle B_a, B_a\rangle_{L^2(\nu)}\\
&=&\tilde{\nu}(a)\|B_a\|_{A^2(u)}^2
\end{eqnarray*}
and consequently, $\tilde{\nu}(a)\leq \|T_\nu\|$.

\end{proof}

We will need the following lemma due to Constantin:
\begin{lem}\label{lemsim}\cite{C2}
Suppose $u$ satisfies condition $B_2$, then $$\|B_a\|_{A^2(u)}^2\sim \frac{\mu(D_r(a))}{(1-|a|)^4}$$
\end{lem}

\begin{pro}\label{p3}
If $\tilde{\nu}$ is bounded on $\D$ then $\nu$ is a $\mu$-Carleson measure, where $\mu=udA$.
\end{pro}

\begin{proof}
Fix $0<r<1$, if $\tilde{\nu}$ is bounded then there exists a constant $C>0$ such that for every $a\in\D$ $$\frac{1}{\|B_a\|_{A^2(u)}^2}\int_{D_r(a)}\frac{1}{|1-\overline{a}z|^4}d\nu(z) \leq C,$$ and consequently
\begin{eqnarray*}
\nu(D_r(a))&\lesssim& C(1-|a|^2)^4\|B_a\|_{A^2(u)}^2\\
           &\sim& \frac{(1-|a|^2)^4\mu(D_r(a))}{(1-|a|)^4}\\
           &\lesssim& \mu(D_r(a))
\end{eqnarray*}
and by theorem \ref{t2} we obtain the result.
\end{proof}

\begin{pro}\label{p4}
Suppose $\nu$ is a $\mu$-Carleson measure, then the Toeplitz operator $T_\nu:A^2(u)\to A^2(u)$ is bounded.
\end{pro}

\begin{proof}
For every function $g\in A_2(u)$ define  the linear operator $g^\ast:A^2(u)\to \mathbb{C}$ by: $g^\ast(f):=\langle f, g \rangle_{L^2(\nu)}.$ Note that if $\nu$ is a $\mu$-Carleson measure, then $$ |g^\ast (f)| \leq \|f\|_{L^2(\nu)}\|g\|_{L^2(\nu)}\leq \|f\|_{A^2(u)}\|g\|_{L^2(\nu)}$$ so $g^\ast \in (A^2(u))^\ast$, $\|g^\ast\|\leq\|g\|_{L^2(\nu)}$ and there exists $\tilde{g}\in A^2(u)$ such that $g^\ast(f)=\langle f, \tilde{g}\rangle_{A^2(u)}$.

We have just proved that for every $g\in A^2(u)$ there exists $\tilde{g}\in A^2(u)$ such that $$\langle f, g\rangle_{L^2(\nu)}=\langle f, \tilde{g}\rangle_{A^2(u)} \qquad \forall f\in A^2(u).$$ In particular, taking $f=K_u(\cdot, a)$ we have that for every $a\in \D$, $$ \tilde{g}(a)=\int g(z)K_u(a,z) d\nu(z)=T_\nu g(a).$$ Hence, $T_\nu(A^2(u)\subset A^2(u)$. Moreover, since $\|\tilde{g}\|_{A^2(u)}=\|g^\ast\|\leq \|g\|_{L^2(\nu)}\leq\|g\|_{A^2(u)}$, then $$\|T_\nu g\|_{A^2(u)}\leq \|g\|_{A^2(u)}$$ and consequently $T_\nu$ is bounded.
\end{proof}

\begin{cor}\label{cinner}
For $T_\nu:A^2(u)\to A^2(u)$ bounded, the following equality holds for every $f$ and $g$ in $A^2(u)$:
\begin{equation}\label{adjtoep}
\langle T_\nu g , f \rangle_{A^2(u)}=\langle g, f \rangle_{L^2(\nu)}
\end{equation}
\end{cor}

\section{Compactness of Toeplitz Operators}

In this section we will characterize compactness of Toeplitz operators in terms of its Berezin transform. We will use the characterization for vanishing Carleson measures given by Constantin.

\begin{defi}
A positive Borel measure $\nu$ on $\D$ is a $\mu$-vanishing Carleson measure if the inclusion operator
$i:A^2(u)\to L^2(\nu)$ is compact.
\end{defi}

\begin{teo}[\cite{C2}, Thm. 3.3]\label{tcompcarchar}
Suppose, $u$ satisfies the $B_2$ condition. Then $\nu$ is a $\mu$-vanishing Carleson measure if and only if for every $0<r<1$, $$\lim_{|a|\to
1^-}\frac{\nu(D_r(a))}{\mu(D_r(a))}=0$$
\end{teo}

\begin{teo}
Suppose $\nu$ is a positive Borel measure on $\D$. Then the following are equivalent:
\begin{itemize}
\item[(a)] $T_\nu$ is compact.
\item[(b)] $\tilde{\nu}(a)\to 0$ as $|a|\to 1^-$.
\item[(c)] $\nu$ is a $\mu$-vanishing Carleson.
\end{itemize}
\end{teo}

\begin{proof}

\noindent (a)$\Rightarrow$ (b)

Notice that by Jensen's inequality we have that \[\mu(D_r(a))\gtrsim (1-|a|^2)^4\left(\int_{D_r(a)}u^{-1}(z)dA(z)\right)^{-1}\] and consequently using Lemma \ref{lemsim} we get \[|b_a(w)|\lesssim |1-\overline{a}z|^{-2} \left(\int_{D_r(a)}u^{-1}(z)dA(z)\right)\longrightarrow 0\] as $|a|\to 1^-$. Thus, the family $\{b_a\}_{a\in\D}$ converges to zero weakly in
$A^2(u)$ as $|a|\to1^-$. We also know from the proof of theorem \ref{tboundbere} that $\displaystyle\tilde{\nu}(a)=\langle T_\nu b_a, b_a \rangle_{A^2(u)}$. Consequently $\tilde{\nu}(a)\leq \|T_\nu b_a\|_{A^2(u)}$ which, since $T_\nu$ is compact, converges to zero as $|a|\to 1^-$.\\

\noindent (b) $\Rightarrow$ (c)

We have that $\tilde{\nu}(a)=\int_{\D}|b_a(z)|^2 d\nu(z)$, and by lemma \ref{lemsim} we obtain that $$\tilde{\nu}(a)\geq \int_{D_r(a)}|b_a(z)|^2 d\nu(z) \sim\frac{\nu(D_r(a)))}{(1-|a|^2)^4 \|B_a\|^2_{A^2(u)}} \sim \frac{\nu(D_r(a))}{\mu(D_r(a))}, $$ then (c) follows from
(b).\\

\noindent (c) $\Rightarrow$ (a)

By corollary \ref{cinner} and Cauchy-Schwarz inequality we have that:

\begin{eqnarray*}
\|T_\nu f\|_{A^2(u)} &=& \sup\{|\langle T_\nu f, g \rangle_{A^2(u)}|: \|g\|_{A^2(u)}=1\}\\
                     &=& \sup\{|\langle f, g \rangle_{L^2(\nu)}|: \|g\|_{A^2(u)}=1\}\\
                     &\leq& \|f\|_{L^2(\nu)}\sup\{\|g\|_{L^2(\nu)}: \|g\|_{A^2(u)}=1\}\\
\end{eqnarray*}
and since $\nu$ is a $\mu$-Carleson measure, then $\|g\|_{L^2(\nu)}\lesssim \|g\|_{A^2(u)}$ for every
$g\in A^2(u)$, so $$\|T_\nu f\|_{A^2(u)} \lesssim \|f\|_{L^2(\nu)} \qquad {\mathrm for ~~all} ~~ f \in
A^2(u).$$

Now, if $f_n$ is a sequence in $A^2(u)$ that converges to zero weakly, then by the compactness of
$i:A^2(u)\to L^2(\nu)$ we have that $\|f_n\|_{L^2(\nu)}$ converges to zero. Therefore, $\|T_\nu
f_n\|_{A^2(u)}$ converges to zero and so, $T_\nu$ is compact on $A^2(u)$.
\end{proof}

We have used the fact that the family of normalized Bergman kernels $\{b_a\}$ converges to zero weakly as $|a|\to 1^-$. This is also true for the family of normalized reproducing kernels of $A^2(u)$.

\begin{pro}
If $u$ satisfies condition $B_2$, then $\displaystyle\frac{K_u(\cdot,a)}{\|K_u(\cdot,a)\|_{A^2(u)}}$ converges to zero weakly in $A^2(u)$.
\end{pro}

\begin{proof}
We use the same notation as in the observation after theorem \ref{t3}. The operator $T:A^2(u)\to A^2(u^{-1})$
maps $\displaystyle\frac{K_u(\cdot,a)}{\|K_u(\cdot,a)\|_{A^2(u)}}$ to
$\displaystyle\frac{B_a}{\|K_u(\cdot,a)\|_{A^2(u)}}$. Also, by lemma \ref{lemsim} we have that
$$\frac{|B_a|}{\|K_u(\cdot,a)\|_{A^2(u)}}\sim \frac{|B_a|}{\|B_a\|_{A^2(u^{-1})}}\lesssim
|B_a|\mu(D_r(a))
$$ which converges to zero uniformly on compact subsets of $\D$. Consequently, the family $\displaystyle
\left\{\frac{B_a}{\|K_u(\cdot, a)\|_{A^2(u)}}\right\}$ is uniformly bounded on $A^2(u^{-1})$ and
converges to zero uniformly on compact subsets of $\D$. Since $A^2(u^{-1})$ is a functional space, then
$\displaystyle\frac{B_a}{\|K_u(\cdot, a)\|_{A^2(u)}}$ converges to zero weakly on $A^2(u^{-1})$.

Now, the operator $T^{-1}:A^2(u^{-1})\to A^2(u)$ is bounded, so
$\displaystyle\frac{K_u(\cdot,a)}{\|K_u(\cdot,a)\|_{A^2(u)}}$ converges to zero weakly in $A^2(u)$.
\end{proof}

\end{document}